\documentclass[oneside,a4paper,10pt]{amsart}
\newcommand{\thetitle}{Random spectrahedra}

\newcommand{\theauthor}{Paul Breiding, Khazhgali Kozhasov and Antonio Lerario}
\newcommand{\citecolor}{black}
\newcommand{\linkcolor}{black}

\newcommand{\txt}[1]{\text{\normalfont{#1}}}
\newcommand{\R}{\mathbb{R}}

\usepackage[utf8]{inputenc}
\usepackage{listings}
\usepackage[T1]{fontenc}
\usepackage[english]{babel}
\usepackage[format=plain, font=footnotesize]{caption}
\usepackage[babel]{csquotes}
\usepackage{tikz}
\usepackage{amsthm}
\usepackage{amsmath,amssymb}
\usepackage[title,titletoc]{appendix}
\usepackage[colorlinks=true, citecolor=\citecolor, linkcolor=\linkcolor]{hyperref}
\usepackage[english]{cleveref}
\usepackage{color,calc,comment}
\usepackage{enumitem}
\usepackage{bbm}
\usepackage{euscript}
\usepackage{mathrsfs}
\usepackage{graphicx}
\usepackage{tkz-euclide, tikz}
\usetikzlibrary{backgrounds,patterns,matrix,calc,arrows,positioning,fit,shapes.geometric,decorations.pathmorphing}
\pgfdeclarelayer{bg}
\pgfdeclarelayer{bg2}    
\pgfsetlayers{bg2,bg,main}
\usepackage{wrapfig}
\setenumerate{labelsep=*, leftmargin=1.5pc}

\newcommand{\HR}{\mathbb{R}}

\newcommand{\sS}{\mathscr{S}}

\newcommand{\cB}{\mathcal{B}}

\newcommand{\cO}{\mathcal{O}}
\newcommand{\cP}{\mathcal{P}}
\newcommand{\cQ}{\mathcal{Q}}
\newcommand{\cR}{\mathcal{R}}
\newcommand{\cS}{\mathcal{S}}

\newcommand{\cZ}{\mathcal{Z}}

\newcommand{\set}[1]{\left\{#1\right\}}
\newcommand{\cset}[2]{\left\{#1\mid #2\right\}}

\newcommand{\tlambda}{\tilde{\lambda}}
\newcommand{\tlambdamin}{\tlambda_{\min}}
\newcommand{\lambdamin}{\lambda_{\min}}

\renewcommand{\d}{\mathrm{ d}}
\renewcommand{\dot}[1]{\overset{\raisebox{-0.25ex}{\scalebox{0.4}{$\bullet$}}}{#1}}

\DeclareMathOperator*{\Prob}{\mathrm{Prob}}

\DeclareMathOperator*{\mean}{\mathbb{E}}

\newcommand{\be}{\begin{equation}}
\newcommand{\ee}{\end{equation}}

\numberwithin{equation}{section}
\numberwithin{figure}{section}
\theoremstyle{plain}
\newcounter{numbering} \numberwithin{numbering}{section}
\newtheorem{thm}[numbering]{Theorem}
\newtheorem{lemma}[numbering]{Lemma}
\newtheorem{prop}[numbering]{Proposition}
\newtheorem{cor}[numbering]{Corollary}
\theoremstyle{definition}

\theoremstyle{remark}

\newtheorem*{rem}{Remark}

\crefname{equation}{}{}
\crefname{equation}{}{}
\crefname{figure}{Figure}{Figures}
\crefname{section}{Section}{Sections}
\crefname{lemma}{Lemma}{Lemmata}
\crefname{prop}{Proposition}{Propositions}
\crefname{thm}{Theorem}{Theorems}
\crefname{cor}{Corollary}{Corollaries}
\crefname{dfn}{Definition}{Definitions}
\crefname{notation}{Notations}{Notations}
\crefname{rem}{Remark}{Remarks}
\crefname{claim}{Claim}{claims}
\begin{document}
\title{\thetitle}
\author{\theauthor}
\thanks{\hspace{-0.5cm}
PB: Max-Planck-Institute for Mathematics in the Sciences (Leipzig), breiding@mis.mpg.de,\\
KK: Max-Planck-Institute for Mathematics in the Sciences (Leipzig), kozhasov@mis.mpg.de,\\
AL: SISSA (Trieste), lerario@sissa.it}
\begin{abstract}
Spectrahedra are affine-linear sections of the cone $\cP_n$ of positive semidefinite symmetric $n\times n$-matrices. We consider random spectrahedra that are obtained by intersecting~$\cP_n$  with the affine-linear space $\mathbbm{1} + V$, where $\mathbbm{1}$ is the identity matrix and $V$ is an $\ell$-dimensional linear space that is chosen from the unique orthogonally invariant probability measure on the Grassmanian of $\ell$-planes in the space of $n\times n$ real symmetric matrices (endowed with the Frobenius inner product). Motivated by applications, for $\ell=3$ we relate the average number $\mean \sigma_n$ of singular points on the boundary of a three-dimensional spectrahedron to the volume of the set of symmetric matrices whose two smallest eigenvalues coincide. In the case of quartic spectrahedra ($n=4$) we show that $\mean \sigma_4 = 6-\frac{4}{\sqrt{3}}$. Moreover, we prove that the average number $\mean \rho_n$ of singular points on the real variety of singular matrices in $\mathbbm{1} + V$ is $n(n-1)$. This quantity is related to the volume of the variety of real symmetric matrices with repeated eigenvalues. Furthermore, we compute the asymptotics of the volume and the volume of the boundary of a random spectrahedron.
\end{abstract}
\maketitle
\section{Introduction}
The intersection of an {affine-linear} subspace $V\subset \textrm{Sym}(n,\mathbb{R})$ of the space of real symmetric matrices with the cone of positive semidefinite symmetric matrices $\mathcal{P}_n$ is called a \emph{spectrahedron}.
The cone over $V\cap \mathcal{P}_n$ is called a \emph{spectrahedral cone}. That is, a spectrahedral cone is the intersection of $\mathcal{P}_n$ with a {linear} subspace of $\textrm{Sym}(n,\mathbb{R})$.

Optimization of a linear function over a spectrahedron is called \emph{semidefinite programming} \cite{Alizadeh1995, Ramana1995}. This is a useful generalization of \emph{linear programming}, i.e., optimization of a linear function over a polyhedron. Problems like finding the smallest eigenvalue of a symmetric matrix or optimizing a polynomial function on the sphere can be approached via semidefinite programming.

The content of this article is a probabilistic study of spectrahedra and spectrahedral cones. We see our results as a contribution to understanding the geometry of feasible sets, which can contribute to better understanding of algorithms in future works. In particular, \cref{main3} on the expected number of singular points on the boundary of a three-dimensional spectrahedron should be highlighted here. For $n=4$ we show that that for our model of a random spectrahedron the expected number of singular points on the boundary is $6-\tfrac{4}{\sqrt{3}}\approx 3.69$, which implies that the appearance of singular points should be anticipated. This is interest for reseachers working in optimization, because the presence of singularities is relevant for semidefinite programming: the set of all linear functions, that when constrained on a three-dimensional spectrahedron attain their maximum in a singular point of the boundary, is an open set, and hence not neglectable.
For example, consider the following cubic spectrahedron (in coordinates):
$$\sS =\cset{(x,z,y)\in \HR^3}{\left(\begin{smallmatrix} 1 & x & y \\ x & 1 & z \\ y & z & 1\end{smallmatrix}\right)\in \cP_n}.$$ The linear function $\psi_w(x,y,z)=\langle w,(x,y,z)\rangle, w\in \HR^3$, constrained on $\sS$ attains its maximum at a point $(x,y,z)\in \partial \sS$ on the boundary of $\sS$ at which the normal cone to $\partial \sS$ contains~$w$. At a singular point of the boundary the normal cone is three dimensional; cf.  shown in \cref{fig0}

Furthermore, we believe that our results will be helpful for testing algorithms when generating random input data. Below we discuss two probabilistic models for a spectrahedron, one of which is empty with high probability, while the other one is non-empty. Therefore, if one wants to validate an algorithm that computes with {non-empty} spectrahedra -- such as for instance the step of following the central path in interior point methods -- one should take the second model for generating random instances for testing.

Besides mentioned applications spectrahedra also appear in modern real algebraic geometry: in \cite{J.WilliamHelton2006} Helton and Vinnikov gave a beautiful characterization of two-dimensional spectrahedra and in \cite{DegIt} Degtyarev and Itenberg described all generic possibilities for the number of singular points on the boundary of a quartic three-dimensional spectrahedron. The reader can also look at the survey \cite{Vinzant}. We thus also see our work as a contribution to \emph{random real algebraic geometry}.

\begin{figure}[t]
{\begin{center}
\includegraphics[height = 3.5cm]{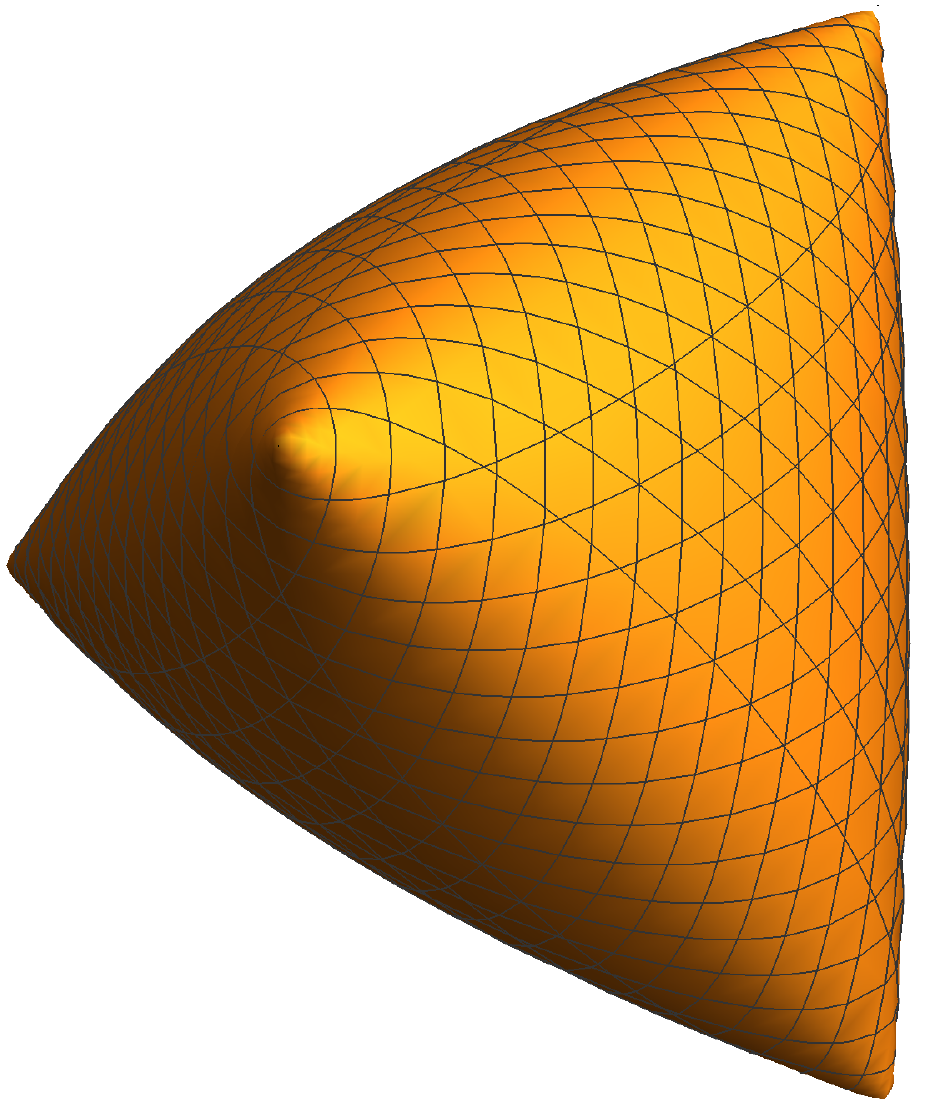} \hspace{2cm}
\includegraphics[height = 3.5cm]{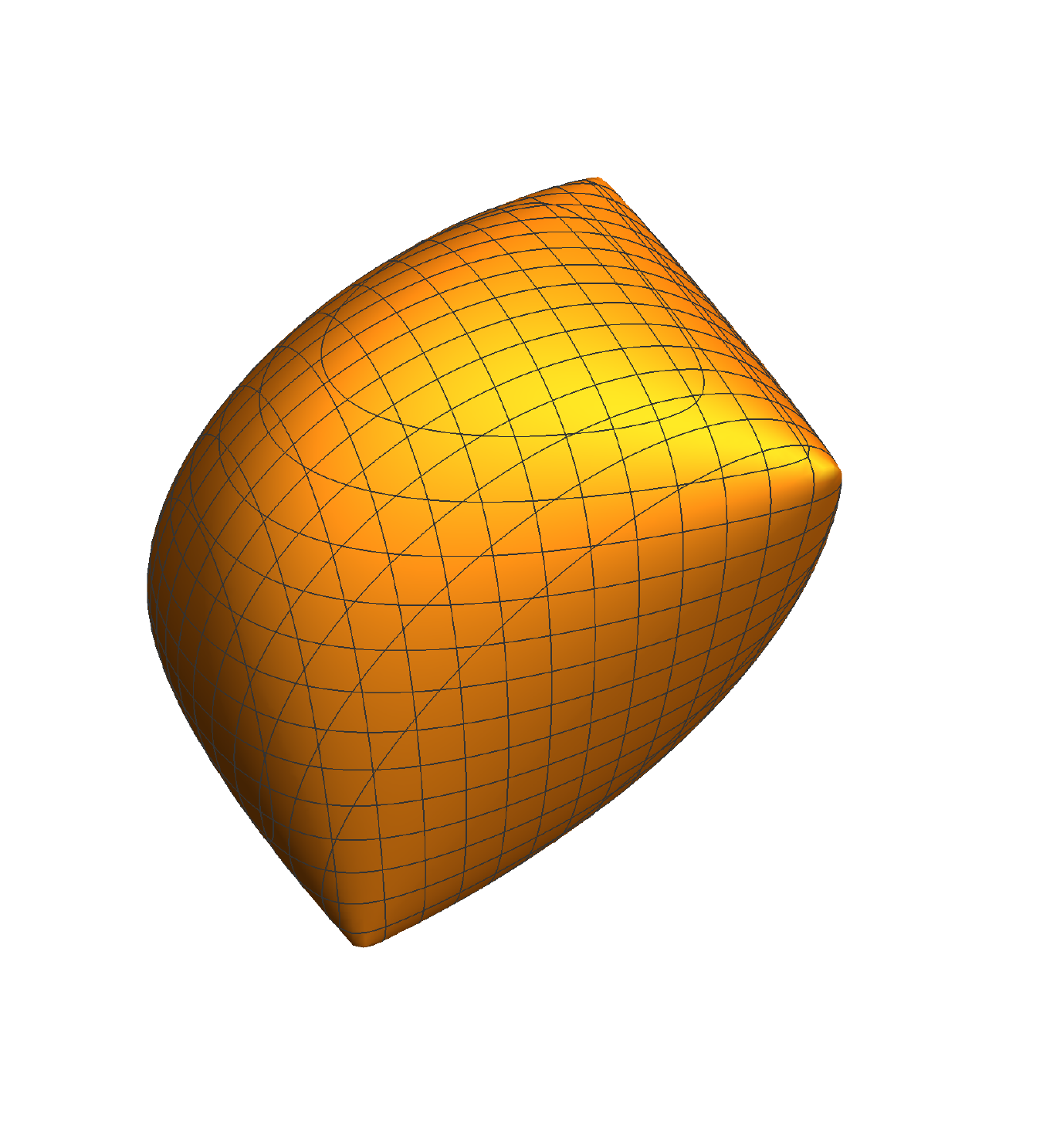}
\end{center}}
\caption{\small \label{fig0}
 On the left is the cubic spectrahedron 
 from the introduction. On the right is a quartic spectrahedron
(in \cite{GrigoriyBlekherman2012} it is called the ``pillow''). The singularities on the boundaries of both spectrahedra are visible. 
}
\end{figure}

%
%
%

\subsection{A possible probabilistic model}
Before discussing the model of this paper, we first describe a first possible natural model for a random spectrahedral cone. This is obtained by taking the intersection of the cone $\mathcal{P}_n$ of positive semidefinite matrices with a linear space drawn from \mbox{\emph{the uniform distribution}} (i.e., the unique orthogonally invariant probability distribution) on \emph{the Grassmannian} $\mathbb{G}_\ell:=\{V \subset \mathrm{Sym}(n, \HR) : \dim V = \ell\}$ of \mbox{$\ell$-dimensional} linear subspaces of $\mathrm{Sym}(n,\HR)$.  
However, this model is of little practical importance as a random spectrahedral cone sampled according to the above distribution is empty with high probability.
Indeed, if, in this model, $\cS_{\ell, n}$ denotes a random spectrahedral cone, then, by \cite[Lemma 4]{gap}, $\mathbb{P}\{ \cS_{\ell,n}\neq \emptyset \}\leq \cO(n^{-c})$, as $n\rightarrow +\infty$ and $\ell$ is fixed, for every $c>0$. In other words, the probability that $\cS_{\ell, n}$ is not empty decays superpolynomially in $n$. Nonempty spectrahedral cones for this model are essentially inaccessible.

\subsection{Our probabilistic model} Motivated by the previous discussion, we propose an alternative model, in which a random spectrahedral cone does not disappear with such an overwhelming probability. To obtain a random spectrahedral cone we intersect $\mathcal{P}_n$ with the linear space $\mathrm{span}\{\mathbbm{1}\} + V$, where  $\mathbbm{1} \in \mathrm{Sym}(n, \HR)$ is the identity matrix and $V$ is chosen from the uniform distribution on $\mathbb{G}_{\ell}$. With this definition a random spectrahedral cone is always nonempty: in fact, the interior of this spectrahedral cone is open and, since it always contains the identity matrix, this open set is also nonempty.

For our probabilisitc study, we consider random spectrahedra in coordinates:
\be\label{spectrahedronQ1} \hat{\sS}_{\ell, n}=\cset{x=(x_1,\ldots,x_\ell)\in \HR^\ell}{\mathbbm{1}+ x_1 Q_1 + \cdots+ x_\ell Q_\ell\in \cP_n},\ee
where the matrices $Q_1, \ldots, Q_\ell$ are independently sampled from \emph{the Gaussian Orthogonal Ensemble} $\textrm{GOE}(n)$ \cite{mehta,tao_matrix}.
We write $Q\sim \textrm{GOE}(n)$ if the joint probability density of the entries of the symmetric matrix $Q\in \textrm{Sym}(n, \R)$ is $\varphi(Q) = 1/C_n \exp(-\Vert Q\Vert_F^2/2),$ where $\Vert Q\Vert_F^2 = \textrm{tr}(Q^2)$ is the square of \emph{the Frobenius norm} and
 $C_n$ is the normalization constant with $\int_{\textrm{Sym}(n, \R)}\varphi(Q)\,\d Q=1$, $\d Q= \prod_{1\leq i<j\leq n} \d Q_{ij}$.  In other words, the entries of $Q$ are centered gaussian random variables, the diagonal entries having variance $1$ and the off-diagonal entries having variance $1/2$. The distribution of the $Q_i$ is orthogonally invariant, and so the linear space $\mathrm{span}\{Q_1,\ldots,Q_\ell\}$ has the same distribution as $V$ above.

In fact, we make two more ``practical'' adjustments to the model \cref{spectrahedronQ1}. First, we rescale the matrices $Q_i$ by the factor $(2n\ell)^{-1/2}$. This serves to balance the order of magnitudes of eigenvalues of the two summands~$\mathbbm{1}$ and~$x_1Q_1+\cdots+x_\ell Q_\ell$. Second, we introduce the following alternative parametrization, which we call a \emph{spherical spectrahedron}:
\be\label{spectrahedronQ} \sS_{\ell, n}=\cset{x=(x_0,\ldots,x_\ell)\in S^\ell}{x_0\,\mathbbm{1}+ \frac{1}{\sqrt{ 2n\ell}} (x_1 Q_1 + \cdots+ x_\ell Q_\ell)\in \cP_n}.\ee
Here, $S^\ell$ denotes the unit sphere in $\HR^\ell$, hence the prefix \emph{spherical}.

In this sense, spherical spectrahedra are just another possible normalization of the spectrahedral cone, in the same way as spectrahedra are the normalization given by fixing one vector in the linear space to be contained in the corresponding affine-linear space. In particular, $\hat{\sS}_{\ell, n}$ is non-empty, if and only if $\sS_{\ell, n}$ is non-empty. Moreover, for three-dimensional random spectrahedra, the number of singular points on the boundaries\footnote{In this paper a ``singular point'' of a spectrahedron is a singular point on the symmetroid surface \eqref{eq:symmetroid} (singular in the sense of algebraic geometry) which belongs to the spectrahedron.} of $\hat{\sS}_{3, n}$ and  $\sS_{3, n}$ almost surely coincide.

In the following, our study focuses on spherical spectrahedra, because the formulas for the volumes are easier to grasp for \cref{spectrahedronQ} than for \cref{spectrahedronQ1}. Furthermore, from now on, for simplicity, we abuse the terminology and use the term ``spectrahedron'' to refer to a spherical spectrahedron.

Our first main result \cref{main1} concerns the expected spherical volume $\vert \sS_{\ell, n}\vert$ of a random spectrahedron $\sS_{\ell, n}$. We show that asymptotically, when both $\ell$ and $n$ tend to $+\infty$, a random spectrahedron on average occupies at least $15\%$ of the sphere $S^\ell$.

Next to the volume of the spectrahedron itself, we are also interested in the volume of its boundary $\partial \sS_{\ell,n}$.
The boundary is contained in the \emph{symmetroid hypersurface}
\begin{equation}\label{eq:symmetroid}\Sigma_{\ell,n}=\{x\in S^{\ell}\mid \det (x_0 \mathbbm{1}+\tfrac{1}{\sqrt{ 2n\ell}} (x_1 Q_1 + \cdots+ x_\ell Q_\ell))=0\}.\end{equation}
The random real algebraic variety $\Sigma_{\ell,n}$ is singular with positive probability when $\ell\geq 3$. Excluding the singular points of $\Sigma_{\ell,n}$ from $\partial \sS_{\ell,n}$, we get a smooth manifold embedded in $S^\ell$, and we define the volume of the boundary of the spectrahedron as the volume of this manifold: $\vert \partial \sS_{\ell,n}\vert := \vert \partial \sS_{\ell,n}\backslash \mathrm{Sing}(\Sigma_{\ell,n})\vert$.

As we noted above, an important feature for applications is the understanding of the structure of singularities on the boundary of the spectrahedron. In the case $\ell = 3$ the singularites are isolated points with probability one (see Corollary \cref{cor:bound}). For the sake of completeness we will study not only the expected number of singular points on $\partial \sS_{3,n}$ but also on the whole symmetroid surface $\Sigma_{3,n}$ (in each case the number of singular points is finite with probability one). By $\sigma_n$ and $\rho_n$ we denote the number of singular points on $\partial \sS_{3,n}$ and on $\Sigma_{3,n}$ respectively.

Summarizing, we will be interested in:
$$\mean | \sS_{\ell,n}|_\mathrm{rel},\quad \mean | \partial \sS_{\ell,n}|_\mathrm{rel},\quad \mean \sigma_n\quad\text{and}\quad  \mean \rho_n,$$
where
$$| \sS_{\ell,n}|_\mathrm{rel} := \frac{| \sS_{\ell,n}|}{\vert S^\ell\vert }\quad\text{and}\quad| \partial \sS_{\ell,n}|_\mathrm{rel} := \frac{|\partial \sS_{\ell,n}|}{\vert S^{\ell-1}\vert }.$$
Our main results on these quantities follow next.
\subsection{Main results}
In the following by $\lambda_{\min}(Q)$ we denote the smallest eigenvalue of a real symmetric matrix $Q$ and we write
	\begin{equation}\label{smallest_EV_scaled}\tlambdamin(Q):=\frac{\lambda_{\min}(Q)}{\sqrt{2n}}
	\end{equation}
for the rescaled smallest eigenvalue.
The following result is proved in \cref{sec:proof1}.
\begin{thm}[Expected volume of the spectrahedron]\label{main1}
Let $F_\ell$ denote the cumulative distribution function of the student's t-distribution with $\ell$ degrees of freedom \cite[Chapter 28]{NormanLJohnson1995} and $\Phi(x)$ denote the cumulative distribution function of the normal distribution \cite[40:14:2]{atlas}. Then\footnote{Here the notation $f(x)=O(g(x))$ means that there exists $C>0$ (called ``the implied constant'') such that $f(x)\leq Cg(x)$ for all $x>0$ large enough.}
\begin{enumerate}
  \item
$\displaystyle \mean |\sS_{\ell,n}|_\mathrm{rel}=\mean_{Q\in \emph{\textrm{GOE}}(n)} \, F_\ell\big(\tlambdamin(Q)\big).$
\item $\displaystyle\mean |\sS_{\ell,n}|_\mathrm{rel}=\Phi(-1)+\cO(\ell^{-1})+\cO(n^{-2/3}),$
  where the implied constants in $\cO(\ell^{-1})$ and $\cO(n^{-2/3})$ are independent of $\ell$ and $n$.
\end{enumerate}
\end{thm}
Note that $\Phi(-1)\approx 0.1587$. This means that asymptotically (when both $n$ and $\ell$ go to $+\infty$) the average volume of a random spectrahedron is at least $15\%$ of the volume of the sphere.
In the following theorem, whose proof is given in \cref{sec:main2}, we are interested in the expected volume of the boundary of a random spectrahedron.
\begin{thm}[Expected volume of the boundary of the spectrahedron]\label{main2} Let $\chi^2_{\ell-1}$ denote the chi-square distribution with $\ell-1$ degrees of freedom \cite[Chapter 18]{NormanLJohnson1995} and define the function
    $$f_{\ell,n}(x) = \frac{\ell}{\ell+x^2}\,\mean\limits_{w\sim \chi^2_{\ell-1}}\,\sqrt{1+\frac{x^2}{\ell} + \frac{w}{2n\ell}}$$
Then:
\begin{enumerate}
  \item
$\displaystyle\mean \vert \partial \sS_{\ell, n}\vert_\mathrm{rel}= \mean\limits_{Q\sim \mathrm{GOE}(n)}\,f_{\ell,n}(\tlambdamin(Q)) .$
\item $\displaystyle\mean \lvert \partial \sS_{\ell, n}\rvert_\mathrm{rel}= 1-\tfrac{1}{2\ell}+\cO(\ell^{-2})+\mathcal{O}(n^{-1/2}),$
where the implied constants in $\cO(\ell^{-2})$ and $\cO(n^{-1/2})$ are independent of $\ell$ and $n$.
\end{enumerate}
\end{thm}

For the average number of singular points $\sigma_n$ on $\partial \sS_{3,n}$ and number of singular points $\rho_n$ on $\Sigma_{3,n}$ the result is more delicate to state. We denote the dimension of $\textrm{Sym}(n, \R)$ by $N:=\frac{n(n+1)}{2}$ and the unit sphere in $\textrm{Sym}(n,\R)$ by $S^{N-1}:=\{Q\in \textrm{Sym}(n,\R)\,|\, \textrm{tr}(Q^2) = 1\}$. Let~$\Delta\subset S^{N-1}$ be the set of symmetric matrices of unit norm and with repeated eigenvalues and let $\Delta_1\subset \Delta$ be its subset consisting of symmetric matrices whose two \textit{smallest} eigenvalues coincide:
\begin{align*}\Delta&:=\{Q\in \textrm{Sym}(n, \R)\cap S^{N-1}\,|\,  \textrm{ $\lambda_i(Q)=\lambda_{j}(Q)$ for some $i\neq j$}\},\\
\Delta_1&:=\{Q\in \textrm{Sym}(n, \R)\cap S^{N-1}\,|\,  \lambda_1(Q)=\lambda_2(Q)\}.
\end{align*}
Note that $\Delta$ and $\Delta_1$ are both semialgebraic subsets of $S^{N-1}$ of codimension two; $\Delta$ is actually algebraic. The following theorem relates $\mean\sigma_n$ and $\mean\rho_n$ to the volumes of $\Delta_1$ and $\Delta$, respectively. We give its proof in Section \ref{sec:nodes}. In the following $|\Delta_1|_\mathrm{rel}:=\tfrac{|\Delta_1|}{|S^{N-3}|}$ and $|\Delta|_\mathrm{rel}:=\tfrac{|\Delta|}{|S^{N-3}|}$.

\begin{thm}[The average number of singular points]\label{main3}
The average number of singular points on the boundary of a random $3$-dimensional spectrahedron $\sS_{3, n}\subset S^3$ equals
\begin{enumerate}\item
$\displaystyle\mean \sigma_n = 2|\Delta_1|_\mathrm{rel}.$
\end{enumerate}
The average number of singular points on the symmetroid $\Sigma_{3,n}\subset S^{3}$ equals
\begin{enumerate}[resume]
\item $\displaystyle \mean \rho_n =  2|\Delta|_\mathrm{rel}.$
\end{enumerate}
\end{thm}

In \cite[Thm. 1.1]{discr_volume} it was proved that $|\Delta|= \tbinom{n}{2}|S^{N-3}|$. This immediately yields the following.
\begin{cor}\label{cor_rho}
$\displaystyle\mean \rho_n = n(n-1).$
\end{cor}

For the expectation of $\sigma_n$ we are lacking such an explicit formula. However combining Theorem \ref{main3} (1) with the formula in \cite[Remark 6]{discr_volume} one obtains
\begin{equation}\label{sigma_n_formula}
\mean \sigma_n  = \frac{2^{n}}{\sqrt{\pi} \,n!}\, \binom{n}{2}\, \, \int_{u\in\HR} \mean_{Q\sim\mathrm{GOE}(n-2)} \big[ \det(Q-u\mathbbm{1})^2  \mathbf{1}_{\{Q-u\mathbbm{1}\in \cP_n\}}  \big]e^{-u^2}\,\d u.
\end{equation}
We expect that $\lim_{n\to \infty}\frac{\mean \sigma_n}{\mean \rho_n}=0,$ but it is difficult to predict how small is $\mean \sigma_n$ compared to~$\mean\rho_n$. The main challenge is to handle the indicator function $\mathbbm{1}_{\{Q-u\mathbbm{1}\in \cP_n\}}$ in the integral above.

Quartic spectrahedra \cite{ORSV2015} are a special case of our study, corresponding to $n=4$. In this case the random symmetroid surface
$$\Sigma_{3,n}=\{x\in S^{3}\mid \det (x_0 \mathbbm{1}+\tfrac{1}{\sqrt{ 2n\ell}} (x_1 Q_1 + x_2 Q_2+ x_3 Q_3))=0\},\; Q_1,Q_2,Q_3\stackrel{\text{i.i.d.}}{\sim} \mathrm{GOE}(n),$$
has degree four, since $\mathbbm{1}, Q_1, Q_2, Q_3\in \textrm{Sym}(4, \R)$. In \cite{DegIt}
Degtyarev and Itenberg proved that all possibilities for $\sigma_4$ and $\rho_4$ are realized by some generic spectrahedra $\sS_{3,4}$ and their symmetroids~$\Sigma_{3,4}$ under the following constraints:
\begin{equation}\label{sigma_4}\sigma_4 \text{ is even\quad and\quad} 0\leq \sigma_4\leq 10;\qquad
\rho_4 \text{ is a multiple of $4$\quad and\quad} 4\leq\rho_4\leq 20.\end{equation}
Degtyarev and Itenberg proved this for the spectrahedron and its symmetroid in projective space, that is why in our condition \eqref{sigma_4} above we have to double their estimates. More generally,
when $n\equiv 2 \bmod 4$ we have the deterministic bound $\rho_n\geq 2$; see \cref{rho_deterministic}.
The inequality $\rho_4\geq 4$ is a notable fact, which has topological reasons (see \cite{FFL, Lax82, FRS84} for related questions).

An ``average picture'' of the Degtyarev and Itenberg result is given in the following proposition.
\begin{prop}[The average number of nodes on the boundary of a quartic spectrahedron]We have \label{prop:nodes} $$\mean \sigma_4 =6-\frac{4}{\sqrt{3}}\approx 3.69\quad\text{and}\quad \mean\rho_4=12.$$
\end{prop}
It would be interesting to understand the distribution of the random variables $\sigma_4, \rho_4$ and compare it with the deterministic picture in \cref{sigma_4}. See also \cref{fig1}.

\begin{figure}[t]
\includegraphics[width = 0.8\textwidth]{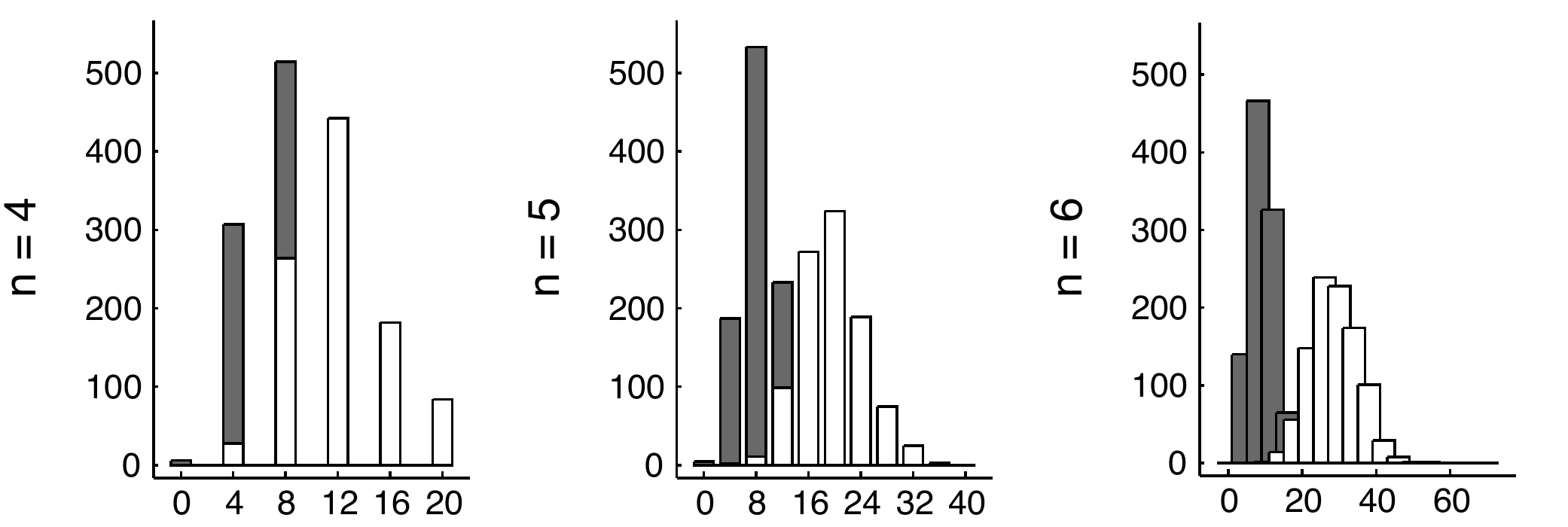}
\caption{The histograms show the outcome of the following numerical experiment. For each $n\in\{4,5,6\}$ we independently sampled 1000 triples $(Q_1,Q_2,Q_2)\sim \mathrm{GOE}(n)^3$. Then, we computed $\rho_n$ and $\sigma_n$ using the \texttt{Julia} \cite{julia} package \texttt{HomotopyContinuation.jl} \cite{HC}. The grey bars show the distribution of $2\sigma_n$ and the white bars the distribution of $\rho_n$. The empirical means are $\overline{\rho_4} = 12.12, \overline{\rho_5}=19.66, \overline{\rho_6}=28.948$
and
$\overline{2\sigma_4} = 7.416, \overline{2\sigma_5}=8.48, \overline{2\sigma_6}=9.3$.\label{fig1}}
\end{figure}

\subsection{Notation} Throughout the article some symbols are repeatedly used for the same purposes: $\textrm{Sym}(n, \R)$ stands for the space of $n\times n$ real symmetric matrices. For the rescaled $Q_i$ we write
$$ R_i=\frac{1}{\sqrt{ 2n\ell}} Q_i.$$
By $\cQ=(Q_1,\ldots,Q_\ell)\in (\textrm{Sym}(n,\mathbb{R}))^\ell$ and $\cR=(R_1,\ldots,R_\ell)\in (\textrm{Sym}(n,\mathbb{R}))^\ell$ we denote a collection of $\ell$ symmetric matrices and its rescaled version respectively. The $k$-dimensional sphere endowed with the standard metric is denoted $S^k$. The symbol $\mathbbm{1}$ stands for the unit matrix (of any dimension). For $x=(x_0,x_1,\ldots,x_\ell)\in S^\ell$ we denote the matrices
\be \label{Q_and_A} Q(x)=\sum_{i=1}^\ell x_iQ_i,\quad  A(x)=x_0\,\mathbbm{1} + \tfrac{1}{\sqrt{2n\ell}}Q(x)\ee
By $\sS_{\ell,n},\, \partial \sS_{\ell,n}$ and $\Sigma_{\ell,n}$ we denote a (random) spectrahedron, its boundary and a symmetroid hypersurface respectively.  Letters $\alpha,\lambda$ and $\mu$ are used to denote eigenvalues and $\tlambda=\tfrac{1}{\sqrt{2n}}\lambda$ stands for the rescaled eigenvalue~$\lambda$.

\subsection{Organization of the article}
The organization of the paper is as follows. In the next section we recall some known deviation inequalities for the smallest eigenvalue of a $\mathrm{GOE}(n)$-matrix. In Sections \ref{sec:proof1}--\ref{sec:nodes} we prove our main theorems. Finally, Section \ref{sec:quartic} deals with the case of quartic spectrahedra.

\section{Deviation inequalities for the smallest eigenvalue}\label{sec:deviation}
In this section we want to summarize known inequalities for the deviation of $\lambda_{\min}(Q)$ from its expected value in the $\textrm{GOE}(n)$ random matrix model. The results that we present are due to \cite{MichelLedoux2010}. Note that in that reference, however, the inequalities are given for the \emph{largest} eigenvalue $\lambda_{\max}(Q)$. Since the $\textrm{GOE}(n)$-distribution is symmetric around $0$, we have $\lambda_{\max}(Q) \sim -\lambda_{\min}(Q)$. Using this we translate the deviation inequalities for $\lambda_{\max}(Q)$ from \cite{MichelLedoux2010} into deviation inequalities for $\lambda_{\min}(Q)$. Furthermore, note that in \cite[(1.2)]{MichelLedoux2010} the variance for the $\mathrm{GOE}(n)$-ensemble is defined differently than it is here: eigenvalues of a random matrix in \cite{MichelLedoux2010} are $\sqrt{2}$ times eigenvalues in our definition.

We express the deviation inequalities in terms of the scaled eigenvalue $\tlambdamin(Q)$, cf. \cref{smallest_EV_scaled}.
The following Proposition is \cite[Theorem 1]{MichelLedoux2010}.
We will not need this result in the rest of the paper directly, but we decided to recall it here because it gives an idea of the behavior of the smallest eigenvalue of a random GOE$(n)$ matrix, in terms of which our theorem on the volume of random spectrahedra is stated.
\begin{prop}\label{prop_tail_bounds_GUE}
For some constant $C>0$, all $n\geq 1$ and $0<\epsilon <1$, we have
$$ \Prob\limits_{Q\in \mathrm{GOE}(n)}\set{\tlambdamin(Q)\leq -(1+\epsilon)} \leq C e^{-C^{-1} n \epsilon^{3/2}}$$
and
$$ \Prob\limits_{Q\in \mathrm{GOE}(n)}\set{\tlambdamin(Q)\geq -(1-\epsilon)} \leq Ce^{-C^{-1}n^2\epsilon^3}.$$
\end{prop}
\cref{prop_tail_bounds_GUE} shows that for large $n$ the mass of $\tlambdamin(Q)$ concentrates exponentially around~$-1$. Thus $\mean \tlambdamin(Q)$ converges to $-1$ as the following proposition shows.
\begin{prop}\label{prop_deviation_from_the_mean} For some constant $C>0$ and all $n\geq 1$ we have
$$\lvert \mean \tlambdamin(Q)+1 \rvert \leq\mean \lvert\tlambdamin(Q) +1\rvert \leq C n^{-2/3}.$$
\end{prop}
\begin{proof}
By \cite[Equation after Corollary 3]{MichelLedoux2010} we have
\begin{equation}\label{prop_deviation_from_the_mean_eq}\limsup\limits_{n\to \infty} 2^p n^{2p/3} \mean\lvert \tlambdamin(Q) +1\rvert^p <\infty.\end{equation}
The assertion follows from monotonicity of the integral:
$\lvert \mean \tlambdamin(Q) +1\rvert  \leq \mean\lvert \tlambdamin(Q) +1\rvert$
and~\eqref{prop_deviation_from_the_mean_eq} with $p=1$.
\end{proof}
\begin{rem}
The distribution of the scaled largest eigenvalue of a $\mathrm{GOE}(n)$ matrix for $n\to \infty$ is known as the \emph{Tracy-Widom distribution} \cite{C.A.Tracy1996}. Suprisingly, this distribution appears in branches of probability that at first sight seem unrelated. For instance, the length of the longest increasing subsequence in a permutation that is chosen uniformly at random in the limit follows the Tracy-Widom distribution \cite{J.Baik1999}. In the survey article \cite{2002math.ph..10034T} Tracy and Widom give an overview of appearances of the distribution in growth processes, random tilings, statistics, queuing theory and superconductors. The present article adds spectrahedra to that list.
\end{rem}
\section{Expected volume of the spectrahedron}\label{sec:proof1}
In this section we prove Theorem \ref{main1}.
\subsection{Proof of Theorem \ref{main1} (1)}
Note that due to the rotational invariance of the standard Gaussian distribution $N(0,1)$ the volume of a spectrahedron $\sS_{\ell,n}$ can be computed as follows:
	\begin{equation}\label{identification}|\sS_{\ell,n}|_{\textrm{rel}}=\frac{\lvert \sS_{\ell,n} \rvert}{\lvert S^\ell\rvert} = \Prob\limits_{\xi_0,\ldots,\xi_\ell \stackrel{\mathrm{iid}}{\sim} N(0,1)} \left\{\xi_0\,\mathbbm{1}+\frac{1}{\sqrt{2n\ell}}\sum_{i=1}^\ell \xi_iQ_i\in \cP_n\right\}
	\end{equation}
Using this and the notations $Q(x)=x_1Q_1 + \cdots + x_\ell Q_\ell$ and $A(x)=x_0 \mathbbm{1} + Q(x)$ from \cref{Q_and_A}
we now write the expectation $\mean |\sS_{\ell,n}|_{\textrm{rel}}$ of the relative volume of the random spectrahedron as:
\begin{equation*} 		\mean \lvert \sS_{\ell,n}\rvert_\mathrm{rel} =\mean\limits_{\cQ\in \textrm{GOE}(n)^\ell}\Prob\limits_{\xi_0,\ldots,\xi_\ell \stackrel{\mathrm{iid}}{\sim} N(0,1)} \left\{\xi_0\,\mathbbm{1}+\frac{1}{\sqrt{2n\ell}}\sum_{i=1}^\ell \xi_iQ_i\in \cP_n\right\}=\mean_{\cQ}\;\mean_{\xi}\; \mathbf{1}_{\{A(\xi) \in \cP_n\}} =: (\star)
\end{equation*}
where $\mathbf{1}_{Y}$  denotes the characteristic function of the set $Y$. Using Tonelli's theorem the two integrations can be exchanged:
\begin{equation}\label{eq:step}(\star)= \mean_{\xi} \; \mean\limits_{\cQ}\; \mathbf{1}_{\set{A(\xi)\in \cP_n} }= \mean_{\xi} \; \Prob\limits_{\cQ\in\textrm{GOE}(n)^{\ell}}\; \set{\xi_0\,\mathbbm{1} + \frac{1}{\sqrt{2n\ell}}Q(\xi)\in \cP_n}.
\end{equation}
For a unit vector $x=(x_1,\dots,x_{\ell})\in S^{\ell-1}$ by the orthogonal invariance of the \textrm{GOE}-ensemble we have $Q(x) \sim \textrm{GOE}(n)$ which leads to
\begin{align*}
  (\star)&= \mean_{\xi}\;\Prob\limits_{Q\in \textrm{GOE}(n)}\; \set{\frac{\xi_0 }{(\xi_1^2+\cdots+\xi_\ell^2)^{1/2}}\,\mathbbm{1} + \frac{1}{\sqrt{2n\ell}}Q\in \cP_n} \\
&= \mean\limits_{Q\in \textrm{GOE}(n)}\; \Prob\limits_{\xi}\set{\frac{\xi_0 \sqrt{\ell}}{(\xi_1^2+\cdots+\xi_\ell^2)^{1/2}}\,\mathbbm{1} + \frac{1}{\sqrt{2n}}Q\in \cP_n},
\end{align*}
where in the second equality we again used Tonelli's theorem. Let us put
	$t_\ell:= \tfrac{\xi_0 \sqrt{\ell}}{(\xi_1^2+\cdots+\xi_\ell^2)^{1/2}}.$
Note that by \cite[(28.1)]{NormanLJohnson1995} the random variable $t_\ell$ follows the Student's t-distribution with $\ell$ degrees of freedom. Since this distribution is symmetric around the origin and $t_{\ell}\mathbbm{1} + \frac{1}{\sqrt{2n}} Q\in \cP_n$ is equivalent to $-t_{\ell}\leq \frac{1}{\sqrt{2n}}\lambdamin(Q)$, we have
$$ (\star) = \mean_{Q\in \textrm{GOE}(n)}\;\Prob\limits_{t_\ell}\set{t_\ell\leq \frac{1}{\sqrt{2n}}\lambdamin(Q)}=\mean_{Q\in \textrm{GOE}(n)}\;F_\ell\big(\tlambdamin(Q)\big),$$
where $F_\ell$ is the cumulative distribution function of the random variable $t_\ell$. This proves Theorem~\ref{main1} (1) since $(\star)=\mean|\sS_{\ell,n}|_{\textrm{rel}}$.\qed
\subsection{Proof of Theorem \ref{main1} (2)}
The random variable $t_{\ell}$ is absolutely continuous. Moreover, its density $F^{\prime}_{\ell}$ is continuous and bounded uniformly in $\ell$ \cite[(28.2)]{NormanLJohnson1995}. In \cref{important_lemma} below we show the (uniform in $\ell$) asymptotic
$\mean \lvert \Sigma_{\ell,n}\rvert_\mathrm{rel} = F_\ell(-1)+\cO(n^{-2/3}),$
where $F_\ell(x)=\Prob\set{t_\ell\leq x}$ and the random variable $t_\ell$ follows the Student's t-distribution with $\ell$ degrees of freedom. By \mbox{\cite[(28.15)]{NormanLJohnson1995}} for fixed $x$ we have
	$F_\ell(x) = \Phi(x) (1+\cO(\ell^{-1})),$
where $\Phi$ is the cumulative distribution function of the standard normal distribution. Plugging in $x=-1$ proves \cref{main1} $(2)$.\qed

\begin{lemma}\label{important_lemma}
Let $f:\HR\to\HR$ be a smooth function such that there exists a constant $c>0$ with $\lVert f'\rVert_\infty \leq c$. Then $\mean_{Q\in \mathrm{GOE}(n)} f\big(\tlambdamin(Q)) = f(-1)+\cO(n^{-2/3})$, and the constant in $\cO(n^{-2/3})$ only depends on $c$ (but not on $f$).
\end{lemma}
\vspace{-0.5cm}
\begin{proof}
Using the intermediate value theorem we write $f(\tlambdamin)=f(-1)+f'(x)(\tlambdamin+1)$ for some $x=x(\tlambdamin)\in\HR$. Taking expectation we obtain
$$\mean f (\tlambdamin)=f(-1)+  \mean\, f'(x)(\tlambdamin+1)\leq f(-1)+  \mean\, \vert f'(x)\vert \,\vert\tlambdamin+1\vert,$$
where the last step is the triangle inequality. This implies, using $\vert f'(x)\vert<c$, that
$$\mean f (\tlambdamin) \leq f(-1) + c\mean|\tlambdamin+1|=f(-1)+\mathcal{O}(n^{-2/3})$$
where the last equality follows from \cref{prop_deviation_from_the_mean} (and is independent of $f$).
\end{proof}

\section{Expected volume of the boundary of the spectrahedron}\label{sec:main2}
In this section we prove Theorem \ref{main2}.
\subsection{Proof of Theorem \ref{main2} (1)}We use the Kac-Rice formula for volume of random manifolds \cite[Theorem 12.1.1]{adler}.
Let $Q(x)=\sum_{i=1}^\ell x_iQ_i$ and $A(x)=x_0\mathbbm{1}+\tfrac{1}{\sqrt{2n\ell}}Q(x)$ be as in \cref{Q_and_A} and denote by $\mu_1(x)\leq \ldots\leq \mu_n(x)$ the ordered eigenvalues of $Q(x)$.  We write the eigenvalues of $A(x)$ as $\alpha_{j}(x)=x_0+\tfrac{1}{\sqrt{2n\ell}}\mu_j(x)$. For a fixed $x$, the density of the random number $\alpha_1(x)$ is denoted $p_{\alpha_1(x)}$. Later we will also need $\lambda_1\leq \cdots\leq \lambda_n$, the eigenvalues of the first matrix $Q_1$.

The set of smooth points of $\partial \sS_{\ell, n}$ is described as follows:
\be (\partial \sS_{\ell, n})_{\textrm{sm}}=\{x\in S^\ell \mid \alpha_1(x)=0, \alpha_1(x)\neq \alpha_2(x)\}\ee
In the following we omit `$x\in S^\ell$' in the notation of sets. By continuity of Lebesgue measure, we have $|\partial \sS_{\ell, n}|=\lim_{\epsilon\to 0}|\{\alpha_1=0\}\cap\{|\alpha_1-\alpha_2|\geq \epsilon\}|$
and, consequently, taking expectation over $\cQ\sim \mathrm{GOE}(n)^\ell$ we have
\begin{align} \mean |\partial \sS_{\ell, n}|&=\mean \lim_{\epsilon\to 0}|\{\alpha_1=0\}\cap\{|\alpha_1-\alpha_2|> \epsilon\}|=\lim_{\epsilon\to 0} \mean |\{\alpha_1=0\}\cap\{|\alpha_1-\alpha_2|> \epsilon\}|;\end{align}
for the last equality we have used monotone convergence to exchange the limit with the expectation.
For a fixed $\epsilon$ the function $\alpha_1$ is smooth on the set $\{|\alpha_1-\alpha_2|>\epsilon\}$ and we can use the Kac-Rice formula\footnote{Here we are applying a generalization of \cite[Theorem 12.1.1]{adler} with the choice $M=S^n$, $f=\alpha_1:S^\ell\to \R$, $h=\alpha_1-\alpha_2:S^\ell\to \R$ (two random fields) and $B=(-\infty, \epsilon)\cup(\epsilon, \infty)\subset \R$. The higher generality comes from the fact that $f^{-1}(0)$ has codimension one in $S^n$; in this, and in the more general case when  $f:M\to \R^k$ with $\dim (M)\geq k$ (the codimension-$k$ case), we have to modify the statement of \cite[Theorem 12.1.1]{adler} as follows. Under the assumption that ``$0$'' is a regular value of the random map $f$ on $h^{-1}(B)$ with probability one, the expectation of the geometric $(\dim(M)-k)$-dimensional volume of $f^{-1}(0)\cap h^{-1}(B)$ equals:
$$\mean|f^{-1}(0)\cap h^{-1}(B)|=\int_{M}\mean\left\{|\mathrm{NJ}f(x)|\cdot \mathbf{1}_{\{h\in B\}}(x)\,\big|\, f(x)=0\right\}\cdot p_{f(x)}(0)\, \d x,$$
where $\mathrm{NJ}f(x)$ denotes the Normal Jacobian of $f$ at $x$; in our case the Normal Jacobian equals the norm of the spherical gradient, i.e. the gradient with respect to an orthonormal basis of $T_{x}{S^\ell}$. The proof of this generalization is essentially identical to the proof of \cite[Theorem 12.1.1]{adler}, with some obvious modifications. A proof of this statement in the Gaussian case, without the constraint ``$h^{-1}(B)$'', can be found in \cite[Theorem 6.10]{azais}.}  \cite[(12.1.4)]{adler} to obtain
\begin{align} \label{eq:step1}\mean |\partial \sS_{\ell, n}|
 =\lim_{\epsilon\to 0}\int_{S^{\ell}}\mean\left\{|\nabla_{S^\ell}\alpha_1(x)|\cdot \mathbf{1}_{\{|\alpha_1( x)-\alpha_2( x)|> \epsilon\}}\,\big|\, \alpha_1(x)=0\right\}\cdot p_{\alpha_1(x)}(0)\, \d x.
\end{align}
Here $\nabla_{S^\ell}$ denotes the gradient with respect to an orthogonal basis of $T_{y}{S^\ell}$, also called the \emph{spherical gradient}, and $p_{\alpha_1(x)}(0)$ is the value of the density of $\alpha_1(x)$ at zero (for fixed $x$).

For a vector $x = (x_0,x_1,\ldots,x_\ell)\in S^{\ell}$ let us denote $\widehat{x}:= (x_1,\ldots,x_\ell)$. Due to the law of adding i.i.d.\ Gaussians we have $\Vert \widehat{x}\Vert^{-1} Q(x)=\Vert \widehat{x}\Vert^{-1}\sum_{i=1}^\ell x_i Q_i \sim \mathrm{GOE}(n)$. Therefore, the probability distribution of $A(x) = x_0 \mathbbm{1} + \tfrac{1}{\sqrt{2n\ell}} Q(x)$
is invariant under rotations preserving the axis through the point $(1,0,\ldots,0)\in S^\ell$. Consequently, the integrand in~\cref{eq:step1} only depends on $y=(x_0, \sqrt{1-x_0^2}, 0,\ldots, 0)$. Moreover, the uniform distribution on $S^\ell$ induces the uniform distribution on the first entry $x_0$. Hence,
\begin{align*}
\mean |\partial \sS_{\ell, n}|=\lim_{\epsilon\to 0}\int_{x_0\in[-1,1]}\mean\left\{|\nabla_{S^\ell}\alpha_1(y)|\cdot \mathbf{1}_{\{|\alpha_1( y)-\alpha_2( y)|>\epsilon\}}\,\big|\, \alpha_1(y)=0\right\} p_{\alpha_1(y)}(0)\,\d x=:(\star)
\end{align*}
Before continuing with the evaluation of $(\star)$ we examine the integrand and the random variables therein. Recall that $\alpha_j(y)$ is the $j$-the eigenvalue of $A(y)$, i.e.,
$ \alpha_1(y)=x_0+\frac{1}{\sqrt{2n\ell}}\,\mu_1(y).$
Note
\begin{equation} \label{alpha_identity}\alpha_1(y)=0\; \text{ if and only if }\;\mu_1(y)=-x_0\sqrt{2n\ell}.
\end{equation}
Let us denote by $u(x)$ the normalized eigenvector of $Q(x)$ associated to the eigenvalue $\mu_1(x)$. The spherical gradient $\nabla_{S^\ell}\alpha_1(x)$ is the projection of the (ordinary) gradient $\nabla_{S^\ell}\alpha_1(x)$ onto $T_{x}{S^\ell}$:
\begin{equation}\label{gradients_relation}\nabla_{S^\ell}\alpha_1(x)=\nabla_{\HR^{\ell+1}}\alpha_1(x)-\langle\nabla_{\HR^{\ell+1}}\alpha_1(x), x\rangle x.
\end{equation}
Using Hadamard's first variation \cite[Section 1.3]{tao_matrix} we can write
\be \nabla_{\HR^{\ell+1}}\alpha_1(x)=\left(1, \frac{1}{\sqrt{2n\ell}}u(x)^TQ_1u(x), \ldots,\frac{1}{\sqrt{2n\ell}}u(x)^TQ_\ell u(x) \right).\ee
Since the $Q_1, \ldots, Q_\ell$ are all symmetric, there is an orthogonal change of basis that makes $Q_1$ diagonal. Then, $Q(y)=\sqrt{1-x_0^2}\,Q_1$ is also diagonal and we can assume that the eigenvector corresponding to $\mu_1(y)$ is $u(y)=(1, 0, \ldots, 0)$. Recall that we have denoted the smallest eigenvalue of $Q_1$ by $\lambda_1$ and note that $\sqrt{1-x_0^2}\,\lambda_1=\mu_1(y)$. Consequently, $u(y)^T Q_1u(y)= \lambda_1$ and therefore we can assume that $\nabla_{\HR^{\ell+1}}\alpha_1(y)$ has the form
\be  \nabla_{\HR^{\ell+1}}\alpha_1(y)=\left(1, \frac{\lambda_1}{\sqrt{2n\ell}},\frac{\xi_2}{\sqrt{2n\ell}}, \ldots,\frac{\xi_\ell}{\sqrt{2n\ell}} \right)\ee
where $\xi_k=(Q_k)_{11}$, $k=2, \ldots, \ell$, are standard independent gaussian variables (recall that in our definition the diagonal entries of a $\mathrm{GOE}(n)$-matrix are standard normal variables). In particular, by \cref{gradients_relation},
\begin{equation}\label{hallo}|\nabla_{S^\ell}\alpha_1(y)|^2
=\vert\nabla_{\HR^{\ell+1}}\alpha_1(y)\vert^2-\langle\nabla_{\HR^{\ell+1}}\alpha_1(y), y\rangle^2
=1+\frac{\lambda_1^{2}}{2n\ell}+\frac{\chi_{\ell-1}^2}{2n\ell}-\Big(x_0+\frac{\mu_1(y)}{\sqrt{2n\ell}}\Big)^2,
\end{equation}
where $\chi_{\ell-1}^2=\sum_{i=2}^\ell\xi_i^2$ is a chi-squared distributed random variable with $\ell-1$ degrees of freedom. Observe that the inner expectation in $(\star)$ is conditioned on the event $\alpha_1(y)=0$. Given this equality we have that $x_0+\tfrac{1}{\sqrt{2n\ell}}\mu_1(y) = 0$. This deletes the last summand in \cref{hallo}. Furthermore, we rewrite the right-hand side of \cref{alpha_identity} as
 \begin{equation}\label{eq_lambda}\frac{\lambda_1}{\sqrt{2n}} = - \frac{x_0\sqrt{\ell}}{\sqrt{1-x_0^2}} .\end{equation}
Recall that $\tlambda=\lambda /\sqrt{2n}$ denotes the scaled eigenvalue and define now the function
\be \label{def_h}h(\tlambda,w)=\sqrt{
1+\frac{\tlambda^{2}}{\ell}+\frac{w}{2n\ell}}.\ee
Denote by $p_{(\tlambda_1,\tlambda_2)}$ the conditional density of $(\tlambda_1,\tlambda_2)$ on $\alpha_1(y)=\alpha_1(x_0,\sqrt{1-x_0^2},0,\ldots,0)=0$ and by $p_{\tlambda_1}$ the marginal density of $\tlambda_1$.
The expectation inside $(\star)$ is with respect to $(\alpha_1(y), \alpha_2(y))$, and we have $\alpha_1(y) = x_0 +\sqrt{\tfrac{1-x_0^2}{\ell}}\,\tlambda_1$. Instead of integrating over $(\alpha_1, \alpha_2)$, we integrate over $(\tlambda_1,\tlambda_2)$ and, using \eqref{def_h}, the expectation becomes:
\be\label{eq:ecco}\mean\left\{|\nabla_{S^\ell}\alpha_1(y)|\cdot \mathbf{1}_{\{|\alpha_1( y)-\alpha_2( y)|>\epsilon\}}\,\big|\, \alpha_1(y)=0\right\}=\mean\limits_{\tlambda_1,w}  \left[h(\tlambda_1,w)\,\mathbf{1}_{\cB(\epsilon)} \mid \tlambda_1= \tfrac{-x_0\sqrt{\ell}}{\sqrt{1-x_0^2}}\right]\ee
where $w\sim \chi^{2}_{\ell-1}$ and $\cB(\epsilon)= \{ \sqrt{\frac{1-x_0^2}{\ell}}\,|\tlambda_1-\tlambda_2|> \epsilon \}$.

We now go back to our integral. First observe that the densities of the eigenvalues at zero are related by:
\be p_{\alpha_1(y)}(0)=\sqrt{ \tfrac{\ell}{1-x_0^2}}\,p_{\tlambda_1}\left(\tfrac{-x_0\sqrt{\ell}}{\sqrt{1-x_0^2}}\right).\ee Using this and \eqref{eq:ecco}, we get
\begin{align*}
 (\star)
 &=\lim\limits_{\epsilon\to 0}\int_{x_0\in[-1,1]}\,\mean\limits_{\tlambda_1,w}  \left[h(\tlambda_1,w)\,\mathbf{1}_{\cB(\epsilon)} \mid \tlambda_1= \tfrac{-x_0\sqrt{\ell}}{\sqrt{1-x_0^2}}\right]\,
\sqrt{ \frac{\ell}{1-x_0^2}}
 \,p_{\tlambda_1}\left(\tfrac{-x_0\sqrt{\ell}}{\sqrt{1-x_0^2}}\right)\,\d x_0.
\end{align*} Using the monotone convergence theorem we can perform the limit within the expectation. Thereafter, $\lambda_2$ does not appear in the variable we take the expected value from. We may omit it to get
\begin{align*}
(\star)
&=\int_{x_0\in[-1,1]}\mean\limits_{\tlambda_1,w} \left[h(\tlambda_1,w)\mid \tlambda_1= \tfrac{-x_0\sqrt{\ell}}{\sqrt{1-x_0^2}}\right]\,\sqrt{ \frac{\ell}{1-x_0^2}}\,p_{\tlambda_1}\left(\tfrac{-x_0\sqrt{\ell}}{\sqrt{1-x_0^2}}\right)\,\d x_0\\
&=\int_{x_0\in[-1,1]}\mean\limits_{\tlambda_1,w} \left[h\left( \tfrac{-x_0\sqrt{\ell}}{\sqrt{1-x_0^2}},w)\right) \right]\,\sqrt{ \frac{\ell}{1-x_0^2}}\,p_{\tlambda_1}\left(\tfrac{-x_0\sqrt{\ell}}{\sqrt{1-x_0^2}}\right)\,\d x_0
\end{align*}
where in the last step we have used the independence of $\tilde{\lambda}_1$ and $w$.
Let us denote by $p_{\chi^{2}_{\ell-1}}(w)$ the density of the $\chi^{2}_{\ell-1}$-random variable. Then we can write
\begin{align*}(\star)&=\int_{x_0\in [-1,1]}\int_{w>0} \sqrt{ \frac{\ell}{1-x_0^2}}\;h\left( \tfrac{-x_0\sqrt{\ell}}{\sqrt{1-x_0^2}},w\right) \;p_{\tlambda_1}\left(\tfrac{-x_0\sqrt{\ell}}{\sqrt{1-x_0^2}}\right) p_{\chi^{2}_{\ell-1}}(w)\, \d w \d x_0.\end{align*}
We use the formula \cref{eq_lambda} to make the change of variables $x_0\mapsto \tlambda_1$ (this trick allows to perform the $x_0$-integration over the smallest rescaled eigenvalue of a GOE$(n)$ matrix). For the differentials we get $\d x_0 = -\tfrac{\ell}{(\ell+\tlambda_1^2)^\frac{3}{2}}\,\d \tlambda_1$ and, writing $\tlambdamin(Q)$ instead of $\tlambda_1$ we obtain:
\begin{align*}(\star)&=\int_\R\left(\mean\limits_{w\sim \chi^2_{\ell-1}}\,\frac{\ell}{\ell+\tilde{\lambda}_1} \,h(\tilde{\lambda}_1,w)\right)p_{\tlambda_1}(\tilde{\lambda}_1)\d\tilde{\lambda}_1\\&= \mean\limits_{Q\sim \mathrm{GOE}(n)}\mean\limits_{w\sim \chi^2_{\ell-1}}\,\frac{\ell}{\ell+\tlambdamin(Q)^2} \,h(\tlambdamin(Q),w)\\
 &= \mean\limits_{Q\sim \mathrm{GOE}(n)}\Bigg[\underbrace{\frac{\ell}{\ell+\tlambdamin(Q)^2} \,\mean\limits_{w\sim \chi^2_{\ell-1}}\,\sqrt{1+\frac{\tlambdamin(Q)}{\ell} + \frac{w}{2n\ell}}}_{=f_{\ell,n}(\tlambdamin(Q))}\Bigg].
\end{align*}
Because $(\star)=\mean \vert \partial \sS_{\ell, n}\vert_\mathrm{rel}$, the assertion follows.\qed
\subsection{Proof of Theorem \ref{main2} (2)}
The overall idea to prove Theorem \ref{main2} (1) is to apply \cref{important_lemma} to the function
  $$ f_{\ell,n}(x) = \frac{\ell}{\ell+x^2}\,\mean\limits_{w\sim \chi^2_{\ell-1}}\,\sqrt{1+\frac{x^2}{\ell}+\frac{w}{2n\ell}}.$$
The derivative of $f_{\ell,n}(x)$ is
	    \begin{equation} \label{bound_deriv}f'_{\ell,n}(x)= \frac{-2x\ell}{(\ell+x^2)^2}\mean\limits_{w\sim \chi^2_{\ell-1}}\,\sqrt{1+\frac{x^2}{\ell}+\frac{w}{2n\ell}} \;+  \frac{\ell}{\ell+x^2}\,\mean\limits_{w\sim \chi^2_{\ell-1}}\,\frac{x}{\ell\sqrt{1+\frac{x^2}{\ell}+\frac{w}{2n\ell}}}
	    \end{equation}
Using the bound $ x\leq \sqrt{\ell}\sqrt{1+\frac{x^2}{\ell}+\frac{w}{2n\ell}}$ we see that the second term in \cref{bound_deriv} is bounded as $0\leq  \frac{\sqrt{\ell}}{\ell+x^2}\,\mean_{w\sim \chi^2_{\ell-1}}\,1 \leq 1$, so that
$g_{\ell,n}(x)  \leq f'_{\ell,n}(x)\leq  g_{\ell,n}(x)+  1$
where
$$g_{\ell,n}(x)=\frac{-2x\ell}{(\ell+x^2)^2}\mean\limits_{w\sim \chi^2_{\ell-1}}\,\sqrt{1+\frac{x^2}{\ell}+\frac{w}{2n\ell}},$$
 and in particular
\begin{equation}	\label{bound_deriv2}
	\vert f'_{\ell,n}(x)\vert \leq \vert g_{\ell,n}(x)\vert +1.
\end{equation}
Next, we bound $\vert g_{\ell,n}(x)\vert$.
Because $\sqrt{a+b}\leq \sqrt{a}+\sqrt{b}$ for non-negative $a$ and~$b$, we have
  $$\sqrt{1+\frac{x^2}{\ell}+\frac{w}{2n\ell}}\leq \sqrt{1+\frac{x^2}{\ell}}+\sqrt{\frac{w}{2n\ell}}.$$
Taking expectation (the expectation of $\sqrt{w}$ can be found, e.g., in \cite{NormanLJohnson1995}) we get
$$\mean\limits_{w\sim \chi^2_{\ell-1}}\,\sqrt{1+\frac{x^2}{\ell}+\frac{w}{n\ell}} \leq \sqrt{1+\frac{x^2}{\ell}}+\frac{1}{\sqrt{2n\ell}} \frac{\Gamma\left(\frac{\ell}{2}\right)}{\Gamma\left(\frac{\ell-1}{2}\right)}.$$
By \cite{Wendel} we have $\tfrac{\Gamma\left(\frac{\ell}{2}\right)}{\sqrt{\ell}\;\Gamma\left(\frac{\ell-1}{2}\right)}\leq 1$ for $\ell\geq 2$ and hence
\begin{equation}\label{boudn_expectation}
\mean\limits_{w\sim \chi^2_{\ell-1}}\,\sqrt{1+\frac{x^2}{\ell}+\frac{w}{2n\ell}} = \sqrt{1+\frac{x^2}{\ell}}+\cO(n^{-1/2}),
\end{equation}
where the constant in $\cO(n^{-1/2})$ does not depend on $\ell$. Thus,
	$$\vert g_{\ell,n}(x)\vert= \frac{2\vert x\vert\ell}{(\ell+x^2)^2}  \left(\sqrt{1+\frac{x^2}{\ell}}+\cO(n^{-1/2})\right),$$
which shows that $\vert g_{\ell,n}(x)\vert$ is bounded uniformly in $x, \ell$ and $n$, i.e., there exists a constant $c\in\HR$ with  $\sup_{\ell,n,x} \vert g_{\ell,n}(x)\vert < c$.
Then, by \cref{bound_deriv2}, $\sup_{\ell,n,x} \vert f_{\ell,n}'(x)\vert \leq c + 1$ and \cref{important_lemma} implies that
	  $$ \mean\limits_{Q\sim\mathrm{GOE}(n)}\, f_{\ell,n}(\tlambdamin(Q)) = f_{\ell,n}(-1) + \cO(n^{-2/3}),$$
where the constant in $\cO(n^{-1/2})$ is independent of $\ell$ and $n$.
Now the value of $f_{\ell,n}$ at $x=-1$ is
\begin{align*} f_{\ell,n}(-1) &= \frac{\ell}{\ell+1}\,\mean\limits_{w\sim \chi^2_{\ell-1}}\,\sqrt{1+\frac{1}{\ell}+\frac{w}{2n\ell}}\\[0.1em] &=\frac{\ell}{\ell+1}\Big(\sqrt{1+\frac{1}{\ell}}+\cO(n^{-1/2})\Big)\\[0.1em] &=\sqrt{\frac{\ell}{\ell+1}}+\cO(n^{-1/2}),
\end{align*}
where the second equality follows from \cref{boudn_expectation}. The assertion is implied by Theorem \ref{main2}(1) and the asymptotic $\sqrt{\tfrac{\ell}{\ell +1}}=1-\tfrac{1}{2\ell}+\cO(\ell^{-2})$
. \qed
\section{The average number of singular points}\label{sec:nodes}
For the study of the average number of singular points on the boundary of a random spectrahedron and on its symmetroid surface we will rely on the following proposition, which implies that this number is generically finite.
\begin{prop}\label{prop:strata}Let $\sS_{\ell, n}^{(k)}$ be the set of matrices of corank $k$ in the spectrahedron $\sS_{\ell, n}$ and~$\Sigma_{\ell, n}^{(k)}$ the set of matrices of corank $k$ in the symmetroid hypersurface $\Sigma_{\ell, n}$. For a generic choice of $\cR=(R_1, \ldots, R_\ell)\in \mathrm{Sym}(n, \R)^\ell$ the sets $\sS_{\ell, n}^{(k)},\Sigma_{\ell, n}^{(k)}\subset S^\ell$ are either empty or semialgebraic of codimension ${k +1\choose 2}$.
\end{prop}
\begin{proof}In the space $\textrm{Sym}(n, \R)$ consider the semialgebraic stratification given by  the corank: $\label{eq:strata} \textrm{Sym}(n, \mathbb{R})=\coprod_{k=0}^{n}\cZ^{(k)},$
where $\cZ^{(k)}$ denotes the set matrices of corank $k$, and the induced stratification on the cone $\cP_n$ of positive semidefinite matrices $\cP_n= \coprod_{k=0}^{n}(\cZ^{(k)}\cap \cP_n)$.  These are Nash stratifications \cite[Proposition 9]{AgrachevLerario} and the codimensions of both $\cZ^{(k)}$ and $\cZ^{(k)}\cap \cP_n$ are equal to ${k+1\choose 2}$.

Consider now the semialgebraic map
\be F:S^\ell \times (\textrm{Sym}(n, \R))^{\ell}\to \textrm{Sym}(n, \R),\;(x, \cR)\mapsto x_0\mathbbm{1}+ x_1 R_1 + \cdots +x_\ell R_\ell.\ee
Then $\Sigma^{(k)}_{\ell, n}=\{x\in S^{\ell}\,|\, F(\cR, x)\in \cZ^{(k)}\}$ and $\sS_{\ell, n}^{(k)}=\{x\in S^{\ell}\,|\, F(\cR, x)\in \cZ^{(k)}\cap \cP_n\}$ and hence they are semialgebraic. We assume that both are non-empty.

We now prove that $F$ is transversal to all the strata of these stratifications. Then the parametric transversality theorem \cite[Chapter 3, Theorem 2.7]{Hirsch} will imply that for a generic choice of $\cR$ the set $\sS_{\ell, n}$ is stratified by the $\sS_{\ell, n}^{(k)}$ and the same for the set $\Sigma_{\ell, n}$. Moreover, when the preimage strata are nonempty, codimensions are preserved. To see that $F$ is transversal to all the strata of the stratifications we compute its differential. At points $(x, \cR)$ with $x\neq e_0 = (1, 0, \ldots, 0)$ we have $D_{(x, \cR)}F(0, \dot\cR)=x_1\dot R_1+\cdots +x_\ell \dot R_\ell$
and the equation $D_{(x, \cR)}F(\dot x, \dot\cQ)=P$ can be solved by taking $\dot x=0$ and $\dot \cR=(0, \ldots, 0, x_i^{-1}P, 0, \ldots, 0)$
where $ x_i^{-1}P$ is in the $i$-th entry and $i$ is such that $x_i\neq 0$ (in other words, already variations in $\cR$ ensure surjectivity of $D_{(x, \cR)}F$). All points of the form $(e_0, \cR)$ are mapped by $F$ to the identity matrix $\mathbbm{1}$ which belongs to the open stratum $\cZ^{(0)}$, on which transversality is automatic (because this stratum has full dimension). This concludes the proof.\end{proof}

The following result on the number of singular points on a generic symmetroid surface is well-known among the experts.
\begin{prop}\label{cor:bound}
For generic $\cR\in \mathrm{Sym}(n, \R)^{3}$ the number of singular points $\rho_n$ on the symmetroid $\Sigma_{3,n}$ and hence the number of singular points $\sigma_n$ on $\partial\sS_{3, n}$ is finite and satisfies
$$\sigma_n \leq \rho_n\leq \frac{n(n+1)(n-1)}{3}.$$
Moreover, for any $n\geq 1$ there exists a symmetroid $\Sigma_{3,n}$ with $\rho_n=\frac{n(n+1)(n-1)}{3}$ singular points on it.
\end{prop}
\begin{proof}
The fact that $\sigma_n\leq \rho_n$ are generically finite follows from \cref{prop:strata} with $k=2$, as remarked before. Observe that $\rho_n$ is bounded by twice (since $\Sigma_{3,n}$ is a subset of $S^3$) the number $\#\text{\normalfont{Sing}} (\Sigma_{3,n}^\mathbb{C})$ of singular points on the complex symmetroid projective surface $$\Sigma_{3,n}^\mathbb{C} = \{x\in \mathbb{C}\mathrm{P}^3 | \det(x_0\mathbbm{1}+ x_1 R_1 + x_2R_2+x_3R_3))=0\}$$
Since $\text{\normalfont{Sing}}(\Sigma_{3,n}^\mathbb{C})$ is obtained as a linear section of the set $\cZ^{(2)}_{\mathbb{C}}$ of $n\times n$ \textit{complex} symmetric matrices of corank two (using similar transversality arguments as in \cref{prop:strata}) we have that generically $\#\text{\normalfont{Sing}} (\Sigma_{3,n}^\mathbb{C})=\deg(\cZ^{(2)}_{\mathbb{C}})$. The latter is equal to $ \tfrac{n(n+1)(n-1)}{6}$; see \cite{HarrisTu}.

Now comes the proof of the second claim, we are thankful to Bernd Sturmfels and Simone Naldi for helping us with this. Consider a generic collection of $n+1$ linear forms $L_1,\dots,L_{n+1}$ in $4$ variables. Then, consider the spectrahedron
$$\sS_{3, n} = \{x\in S^3 : \mathrm{diag}(L_1(x), \ldots,L_{n}(x)) + L_{n+1}(x) \,ee^T \in \cP_n\},$$
where $e=(1,\ldots,1)$. The corresponding symmetroid $\Sigma_{3,n}$ is defined by the equation
$$\det(\mathrm{diag}(L_1(x), \ldots,L_{n}(x)) + L_{n+1}(x) \,ee^T) = \sum_{i=1}^{n+1} \prod_{j\neq i} L_i(x) = 0.$$
(the equality can, for instance, be shown by using the determinant rule for rank-one updates). Therefore the singular points on $\Sigma_{3,n}$ are given by the equations
$$\sum_{i=1}^{n+1} \prod_{j\neq i} L_i(x) = 0 \quad \text{ and }\quad  \sum_{i=1}^{n+1} \sum_{j\neq i} \prod_{k \not \in  \{i,j\}} L_k(x) = 0.$$
The triple intersections of the hyperplanes $L_1,\dots,L_{n+1}$ produce $2{n+1\choose 3}=\frac{(n+1)n(n-1)}{3}$ solutions in $S^3$ for those equations. They are $\frac{(n+1)n(n-1)}{3}$ singular points on $\Sigma_{3,n}$.
\end{proof}
We now prove Theorem \ref{main3}.
\subsection{Proof of Theorem \ref{main3}} We prove Theorem \ref{main3} (2). The proof of \cref{main3} (1) is completely analogous.

By Proposition \ref{prop:strata}, for a generic choice of $\cR = (R_1,R_2,R_3)\in \mathrm{Sym}(n,\HR)^3$ the singular locus of $\Sigma_{3,n}$ is given by the matrices of corank $2$ in $\Sigma_{3,n}$:
\begin{align}\label{eq1:duality}
	 \Sigma_{3,n}^{(2)} = \txt{Sing}(\Sigma_{3,n}),
\end{align}
and the set $\Sigma_{3,n}^{(2)}$ is finite.

We first show that the number $\#(\mathrm{Sing}(\Sigma_{3,n})$ of singular points on the symmetroid surface defined by $\cR=(R_1,R_2,R_3)$ coincides with the number $\#(V\cap \Delta)$ of matrices in the linear space $V=\mathrm{span}\{R_1,R_2,R_3\}$ meeting $\Delta$. For this observe that
 \begin{align}\label{eq2:duality}
\lambda_i(x_0\mathbbm{1}+R(x)) = x_0+\lambda_i(R(x)),\ i=1,\dots,n,
\end{align}
where we denote $R(x)=x_1R_1+x_2R_2+x_3R_3$. Thus, if two eigenvalues of $R(x)$, say $\lambda_i(R(x))$ and $\lambda_j(R(x))$, coincide, then, due to  \eqref{eq2:duality} we have
\begin{align*}
\lambda_i(x_0\mathbbm{1}+R(x)) = \lambda_j(x_0\mathbbm{1}+R(x)) = 0 \text{ for } x_0 = -\lambda_i(R(x)).
\end{align*}
Therefore, by \cref{eq2:duality},
\begin{align*}
\frac{(-\lambda_i(R(x)),x_1,x_2,x_3)}{\sqrt{\lambda_i(R(x))^2+x_1^2+x_2^2+x_3^2}} \in \sS_{3,n}^{(2)}=\txt{Sing}(\partial\sS_{3,n})
\end{align*}
is a singular point of $\partial\sS_{3,n}$. Vice versa, if $(x_0,x_1,x_2,x_3)\in \txt{Sing}(\sS_{3,n})$ we have, by \cref{eq2:duality}, $\lambda_1(R(x))=\lambda_2(R(x))$ and $(x_1,x_2,x_3)\neq 0$. Hence, $R(x)/\Vert R(x)\Vert \in \Delta$. Moreover, one can easily see that the described identification is one-to-one.

When $Q_1, Q_2, Q_3\in \mathrm{Sym}(n,\HR)$ are $\mathrm{GOE}(n)$-matrices and $R_i = Q_i/\sqrt{6n}, i=1,2,3$ the space $V=\mathrm{span}\{R_1,R_2,R_3\}$ is uniformly distributed in the Grassmanian $\mathbb{G}_3$. Applying the integral geometry formula \cite[p. 17]{Howard} to $\Delta\subset S^{N-1}$ and the random space $V\subset \mathrm{Sym}(n,\HR)$ we write
\
\begin{align}\label{eq: integral geometry}
  \mean \#(V\cap \Delta) = 2\frac{|\Delta|}{|S^{N-1}|} = 2|\Delta|_{\mathrm{rel}}
\end{align}
From the above it follows that with probability one $\#(V\cap \Delta) = \#(\mathrm{Sing}(\Sigma_{3,n}))= \rho_n$. Thus $\mean \#(V\cap \Delta) = \mean \rho_n$ which combined with \eqref{eq: integral geometry} completes the proof of Theorem \ref{main3} (1).\qed

From this proof we get the following deterministic corollary.
\begin{cor}\label{rho_deterministic}
If $n\equiv 2 \bmod 4$, then $\rho_n\geq2$.
\end{cor}
\begin{proof}
As was shown in the proof, $\rho_n = \#(V\cap \Delta)$, where $V$ is the 3-dimensional linear space $V=\mathrm{span}\{R_1,R_2,R_3\}$. Lax \cite{Lax82} proved that every 3-dimensional linear subspace of $\mathrm{Sym}(n,\mathbb{R})$ contains at least one point in $\Delta$. This implies $\rho_n\geq2$ (because, points on $S^{N-1}$ come in pairs of antipodal points). See also \cite{FRS84} for other values of $n$ with $\rho_n=\#(V\cap \Delta)\geq 2$.
\end{proof}
\section{Quartic Spectrahedra}\label{sec:quartic}
In this section we prove \cref{prop:nodes}: the identity $ \mean\rho_4=12$ follows immediately from \cref{cor_rho} for $n=4$.
For the other identity we apply \cref{sigma_n_formula} and the formula \cite[(3.2)]{discr_volume}:
\begin{align*}\mean \sigma_4
&=\frac{2^4}{\sqrt{\pi}}\frac{1}{4!}{4\choose 2} \, \int_{u\in\HR} \mean_{Q\sim\mathrm{GOE}(2)} \big[\mathbf{1}_{\{Q-u\mathbbm{1}\in \cP_n\}}\,\det(Q-u\mathbbm{1})^2 \big]e^{-u^2}\,\d u\\
&=\frac{4}{\sqrt{\pi}} \int_{\R}\left(\frac{1}{Z_2}\int_{\R^2}(\lambda_1-u)^2(\lambda_2-u)^2|\lambda_1-\lambda_2|e^{-\frac{\lambda_1^2+\lambda_2^2}{2}}\mathbf{1}_{\{\lambda_1-u>0, \lambda_2-u>0\}}d\lambda_1d\lambda_2\right)e^{-u^2}\,\d u\\
&=(\star),
\end{align*}
where $Z_2= 4\sqrt{\pi}$ is the normalization constant for the density of eigenvalues of a $\mathrm{GOE}(2)$-matrix (see \cite[(3.1)]{discr_volume}).

We now apply the change of variables $\alpha_1=\lambda_1-u, \alpha_2=\lambda_2-u$ in the innermost integral, obtaining:
\begin{align*}(\star)&=\frac{4}{\sqrt{\pi}} \int_{\R}\left(\frac{1}{Z_2}\int_{\R_+^2}(\alpha_1\alpha_2)^2|\alpha_1-\alpha_2|e^{-\frac{\alpha_1^2+\alpha_2^2}{2}}e^{-u^2 -u(\alpha_1+\alpha_2)}\,\d\alpha_1\,\d\alpha_2\right)e^{-u^2}\,\d u\\
&=\frac{4}{\sqrt{\pi}}\frac{1}{Z_2}\int_{\R_+^2} (\alpha_1\alpha_2)^2|\alpha_1-\alpha_2|e^{-\frac{\alpha_1^2+\alpha_2^2}{2}}\left(\int_\R e^{-2u^2 -u(\alpha_1+\alpha_2)}du\right)\,\d\alpha_1\,\d\alpha_2\\
&= \frac{1}{\pi}\int_{\R_+^2} (\alpha_1\alpha_2)^2|\alpha_1-\alpha_2|e^{-\frac{\alpha_1^2+\alpha_2^2}{2}}\left(\sqrt{\frac{\pi}{2}} e^{\frac{(\alpha_1+\alpha_2)^2}{8}}\right)\,\d\alpha_1\,\d\alpha_2\\
&=\frac{1}{\sqrt{2\pi}}\int_{\R_+^2} (\alpha_1\alpha_2)^2|\alpha_1-\alpha_2|e^{-\frac{\alpha_1^2+\alpha_2^2}{2}+\frac{(\alpha_1+\alpha_2)^2}{8}}\,\d\alpha_1\,\d\alpha_2\\
&=\frac{2}{\sqrt{2\pi}}\int_{\R_+^2\cap\{\alpha_1< \alpha_2\}} (\alpha_1\alpha_2)^2|\alpha_1-\alpha_2|e^{-\frac{\alpha_1^2+\alpha_2^2}{2}+\frac{(\alpha_1+\alpha_2)^2}{8}}\,\d\alpha_1\,\d\alpha_2.
\end{align*}
In the last equality we have used the fact that the integrand is invariant under the symmetry $(\alpha_1, \alpha_2)\mapsto (\alpha_2, \alpha_1).$
Consider now the map $F:\R_+^2\cap\{\alpha_1< \alpha_2\}\to \R[x]\simeq \R^2$ given by
$$F(\alpha_1, \alpha_2)=(x-\alpha_1)(x-\alpha_2)=x^2-(\alpha_1+\alpha_2)x+\alpha_1\alpha_2.$$
Essentially, $F$ maps the pair $(\alpha_1, \alpha_2)$ to a monic polynomial of degree two whose ordered roots are $(\alpha_1, \alpha_2).$ Observe that $F$ is injective on the region $\R_+^2\cap\{\alpha_1< \alpha_2\}$ with never-vanishing Jacobian $|JF(\alpha_1, \alpha_2)|=|\alpha_1-\alpha_2|$. What is the image of $F$ in the space of polynomials $\R[x]$? Denoting by $a_1, a_2$ the coefficients of a monic polynomial $p(x)=x^2-a_1x+a_2\in \R[x]$, we see first that the conditions $\alpha_1, \alpha_2>0$ imply $a_1, a_2>0$. Moreover the polynomial $p(x)=F(\alpha_1, \alpha_2)$ has by construction real roots, hence its discriminant $a_1^2-4a_2$ must be positive.

Viceversa, given $(a_1, a_2)$ such that $a_1,a_2>0$ and $a_1^2-4a_2>0$, the roots of $p(x)=x^2-a_1x+a_2$ are real and positive. Hence,
$ F(\R_+^2\cap\{\alpha_1< \alpha_2\})=\{(a_1, a_2)\in \R^2\,|\, a_1,a_2>0, a_1^2-4a_2>0\}.$
Thus we can write the above integral as
\begin{align*}(\star)&=\frac{2}{\sqrt{2\pi}}\int_0^{\infty}\int_{0}^{\frac{a_1^2}{4}} a_2^2e^{-\frac{a_1^2-2a_2}{2}+\frac{a_1^2}{8}}\,\d a_2\d a_1
=\frac{2}{\sqrt{2\pi}}\int_0^{\infty} e^{-\frac{3 a_1^2}{8}}\left(\int_{0}^{\frac{a_1^2}{4}} a_2^2 e^{a_2}da_2\right)\,\d a_1
\end{align*}
and performing elementary integration we obtain $(\star)= \mean \sigma_4=6-\frac{4}{\sqrt{3}}$.

\subsection*{Acknowledgements}The authors wish to thank P. B\"urgisser, S. Naldi, B. Sturmfels and three anonymous referees for helpful suggestions and remarks on the paper.

\bibliographystyle{plain}
\bibliography{literature}

\end{document}